\documentclass[A4paper,oneside,11pt]{amsart}
\usepackage{amssymb,amsthm,amsmath,marvosym,amsfonts,fontenc}
\usepackage{enumerate}
\usepackage{comment}
\usepackage{enumerate,marginnote}
\usepackage{graphicx}
\usepackage{enumerate}
\usepackage{color}

\DeclareMathOperator{\ex}{ex}

\def\qed{\hfill $\square$}

\def\C{\mathcal{C}}

\def\HH{\mathcal{H}}
\def\eps{\varepsilon}
\let \phi=\varphi
\newcommand{\hm}[1]{\leavevmode{\marginpar{\tiny%
$\hbox to 0mm{\hspace*{-0.5mm}$\leftarrow$\hss}%
\vcenter{\vrule depth 0.1mm height 0.1mm width \the\marginparwidth}%
\hbox to 0mm{\hss$\rightarrow$\hspace*{-0.5mm}}$\\\relax\raggedright #1}}}

\begin{document}
\allowdisplaybreaks[2]

\newtheorem{theorem}{Theorem}[section]
\newtheorem{cor}[theorem]{Corollary}
\newtheorem{lemma}[theorem]{Lemma}
\newtheorem{fact}[theorem]{Fact}
\newtheorem{property}[theorem]{Property}
\newtheorem{corollary}[theorem]{Corollary}
\newtheorem{proposition}[theorem]{Proposition}
\newtheorem{claim}[theorem]{Claim}
\newtheorem{conjecture}[theorem]{Conjecture}
\newtheorem{definition}[theorem]{Definition}
\theoremstyle{definition}
\newtheorem{example}[theorem]{Example}
\newtheorem{remark}[theorem]{Remark}

\title{Maximum number of edge colorings avoiding rainbow copies of $K_4$}
\author[H. H\`{a}n]{Hi\d{\^e}p H\`{a}n}
\address{Departamento de Matem\'{a}tica y Ciencia de la Computaci\'{o}n, Universidad de Santiago de Chile --  Las Sophoras, 173, Estaci\'{o}n Central, Santiago, Chile}
\email{hiep.han@usach.cl}

\author[C. Hoppen]{Carlos Hoppen}
\address{Instituto de Matem\'{a}tica e Estat\'{i}stica, Universidade Federal do Rio Grande do Sul -- Avenida Bento Gon\c{c}alves, 9500, 91501--970 Porto Alegre, RS, Brazil}
\email{choppen@ufrgs.br}

\author[N. M. M\"{u}ller]{Nicolas Moro M\"{u}ller}
\address{Instituto Federal do Rio Grande do Sul, Campus Caxias do Sul -- Rua Avelino Ant\^{o}nio de Souza, 1730, 95043-700, Caxias do Sul, RS, Brazil}
\email{nicolas.muller@caxias.ifrs.edu.br}

\author[D. R. Schmidt ]{Dionatan Ricardo Schmidt}
\address{Universidade Federal do Pampa, Campus Alegrete -- Avenida Tiaraju, 810, 97546--550, Alegrete, RS, Brazil}
\email{dionatanschmidt@unipampa.edu.br}

\thanks{C.~Hoppen was partially supported by CNPq (315132/2021-3 and 408180/2023-4) and FAPERGS (21/2551-0002053-9).}
\thanks{Hi\d{\^e}p H\`an was supported by the ANID Regular grant 1231599 and by ANID Basal Grant CMM FB210005.}

\begin{abstract}
In this paper we show that for $r\geq 12$ and any sufficiently large $n$-vertex graph $G$ the number of $r$-edge-colorings of $G$ with no rainbow $K_4$ is at most $r^{\ex(n,K_4)}$, where $\ex(n,K_4)$ denotes the Tur\'{a}n 
number of $K_4$. Moreover, $G$ attains equality  if and only if it is the Tur\'{a}n graph $T_3(n)$. 

The bound on the number of colors $r\geq 12$ is best possible. It improves upon a result of H. Lefmann,  D.A. Nolibos, and the second author who showed the same result for $r \geq 5434$
and it confirms a conjecture by
Gupta, Pehova, Powierski and Staden.

\end{abstract}

\maketitle

\section{Introduction}
Given a graph $G$ and integers $r,k \geq 3$, an $r$-coloring $\phi\colon E(G)\to[r]$ of the edges of $G$ is  \emph{free of rainbow $K_k$} if there is no copy of the clique $K_k$ in $G$ whose edges are colored
by pairwise distinct colors. If $k=3$, these $r$-colorings are known as \emph{Gallai colorings}~\cite{GS04}. For general $k$, the number of such $r$-colorings of $G$ is denoted by $\rho_{r,k}(G)$  and in this paper we 
study  $\rho_{r,k}(n)=\max \{\rho_{r,k}(G)\colon {G\subset K_n}\}$, the maximum number of rainbow-$K_k$-free colorings an $n$-vertex graph can have. Moreover,
 we are interested in the structure of the graphs which attain this maximum. 
 This is a natural and popular variant of the classical Erd\H{o}s-Rothschild problem, which studies the same problem for colorings with no monochromatic $K_k$, see below for more details.

If $G$ itself is $K_k$-free, then $\rho_{r,k}(n)\geq \rho_{r,k}(G)=r^{e(G)}$ and thus a straightforward lower bound for $\rho_{r,k}(n)$ is obtained by 
considering the largest $K_k$-free graph, which was determined by Tur\'an's theorem~\cite{turan} 
to be the complete $(k-1)$-partite graph with partition classes as equal as possible. This commonly called \emph{Tur\'an graph} and its number of edges are denoted by $T_{k-1}(n)$ and $\ex(n,K_k)$, respectively.
Then 
\begin{equation}\label{eq:bounds}
\rho_{r,k}(n)\geq \rho_{r,k}(T_{k-1}(n))= r^{\ex(n,K_k)},
\end{equation} 
and computing the value of $\rho_{r,k}(n)$ was first posed as a problem in~\cite{BHS17}.
In~~\cite{rainbow_kn} Lefmann, Odermann, and the second author showed that the inequality in~\eqref{eq:bounds} holds with equality if the number of colors is sufficiently large.
\begin{theorem}\cite[Theorem 1.2]{rainbow_kn}  \label{thm_Kk} 
For every $k \geq 3$ and $r \geq \binom{k}{2}^{8k-4}$ there is $n_0$ such that every graph of order $n > n_0$ has at most $r^{\ex(n,K_{k})}$ rainbow-$K_{k}$-free $r$-colorings. Moreover, the Tur\'{a}n graph $T_{k-1}(n)$ is the only graph on $n$ vertices which achieves equality.
\end{theorem}
As another lower bound on $\rho_{r,k}(n)$ note that $\big(\binom k2-1\big)$-colorings of $K_n$ contain no rainbow copies of~$K_k$. For $r$ sufficiently small compared to $k$ we have  $\binom k2-1>r^{\left(1-\frac1{k-1}\right)}$ 
and thus  $\rho_{r,k}(K_n)>\rho_{r,k}(T_{k-1}(n))$.
Hence, while the statement of  Theorem~\ref{thm_Kk} is expected to hold for much smaller values of $r$, it is not true for arbitrary~$r$, and determining the smallest~$r$ for which the statement  holds is an interesting and challenging problem.  

For triangles, i.e.~$k=3$, this problem has been addressed and the bound in Theorem~\ref{thm_Kk} has been improved by the same  authors in~\cite{rainbow_triangle}  to $r \geq 5$. Later
this was lowered by Balogh and Li~\cite{BL19}  to $r\geq 4$ which is best possible.  %
Regarding the general case, Lefmann, Nolibos and second author~\cite{multipattern} have improved the  bound in Theorem~\ref{thm_Kk} to $r \geq (k^2/4)^{4k}$ for all $k\geq 4$ and for $k=4$, they 
obtained the  bound $r \geq 5434$. This seemed to be still far from the correct value and very recently, Gupta, Pehova, Powierski, and Staden~\cite[Problem 8.3]{GPPS25+} conjectured that the 
correct bound should be $r \geq 12$, 
which would be best possible. 
Our main result is a proof of this conjecture.
\begin{theorem}  \label{thm_main} 
For every integer $r \geq 12$, there is $n_0$ such that every graph of order $n > n_0$ has at most $r^{\ex(n,K_{4})}$ rainbow-$K_4$-free $r$-colorings. Moreover, the Tur\'{a}n graph $T_3(n)$ is the only graph on $n$ vertices for which equality is achieved.
\end{theorem}
The proof of the theorem may be split into two steps, the first one proving a stability result and the second one proving the exact result.  
Fortunately, the exact part has already been shown in a more general setting in~\cite{rainbow_kn} and we will make use of it here.
A graph $G$ is called \emph{$t$-close  to being~$k$-partite} if it admits a partition $V(G)=V_1\cup\dots\cup V_k$ such that $\sum_{i\in[k]}e(V_i)\leq t$.
\begin{theorem}\cite[Lemma 4.4]{rainbow_kn} \label{exact} 
Let  $k \geq 3$ and $r>ek$\footnote{Here $e$ stands for Euler's number.}. Suppose that for every $\delta>0$ there is an $n_0$ such that any  graph $G$ on $n\geq n_0$ vertices with $\rho_{r,k}(G)\geq r^{\ex(n,K_k)}$ is $\delta n^2$-close to being
$(k-1)$-partite. Then, for $n$ sufficiently large, we have 
\[\rho_{r,k}(n)=r^{\ex(n,K_k)},\]
and the only graph which achieves the maximum is the Tur\'an graph $T_{k-1}(n)$.
\end{theorem}

We note that the original version of Theorem~\ref{exact} is  more general and also applies to color patterns of $K_k$ with a vertex incident to $k-1$ edges of distinct colors.
Since $4e<12$, our main result, Theorem~\ref{thm_main}, follows from the following.
\begin{theorem}\label{thm_main2}
For any integer $r\geq 12$ and any $\delta>0$, there is $n_0$ with the following property. 
If~$G$ is a graph on $n>n_0$ vertices with $\rho_{r,4}(G)\geq r^{\ex(n,K_4)}$, then $G$ is $\delta n^2$-close to being tripartite, i.e., there is a partition 
$V(G)=V_1 \cup V_2 \cup V_3$ such that $e(V_1)+e(V_2)+e(V_3)< \delta n^2$.
\end{theorem}
Our approach to Theorem~\ref{thm_main2} is inspired by the solution of the problem for triangles by Balogh and Li~\cite{BL19} and is based on multi-color hypergraph containers.

\subsection*{The  Erd\H{o}s-Rothschild problem and further related works}
As mentioned previously, the problem studied here is a natural and popular variation of the classical problem posed by Erd\H{o}s and Rothschild~\cite{Erd74} who asked
for the maximum number of $2$-edge-colorings with no monochromatic copies of the $K_k$.
Confirming their conjecture, Yuster~\cite{yuster} (for $k=3$), and Alon, Balogh, Keevash, and Sudakov~\cite{ABKS} (for $k\geq 4$) showed that
this maximum is attained  by the Tur\'an graph $T_{k-1}(n)$ and by this graph only.
Actually, the results in~\cite{ABKS} show that the same conclusion holds for $r=3$ colors and for all $k \geq 2$ (see also Jim\'{e}nez and the first author~\cite{hanjime}). 
For $r\geq 4$ the phenomenon does not persist and, while the problem and its variations enjoy a vivid interest, 
progress has been slow~\cite{botler,PS2021,PS2022,yilma}. There has been interest in colorings avoiding monochromatic copies of other graphs, see for instance \cite{CGM20,forbm,linear}.

Balogh~\cite{Bal06} was the first to consider 2-edge-colorings that avoid complete graphs with a prescribed coloring that is not monochromatic. 
The variant of this problem where we are interested in $r$-edge-colorings that avoid a given color pattern, which became known as the generalized Erd\H{o}s-Rothschild problem, has been proposed by Benevides, Sampaio, and the second author~\cite{BHS17}.  In particular, they showed that for almost all patterns of complete graphs, the maximum number of colorings is achieved by a complete multipartite graph. We remark that their result notably does not apply to the monochromatic pattern. Moreover, it does not imply that these extremal complete multipartite graph must be balanced. This has been further extended to a setting where the colorings must avoid a set of patterns, see~\cite{multipattern}. A recent paper by Gupta, Pehova, Powierski, and Staden~\cite[Section 3]{GPPS25+} makes progress in this general direction by associating the solution of the generalized Erd\H{o}s-Rothschild problem with the solution of a related optimization problem, from which the authors obtain precise results for some families of forbidden patterns.

To conclude the introduction, we mention that Erd\H{o}s-Rothschild-type problems have also been studied for other combinatorial structures, such as hypergraphs~\cite{CDT2018,kneser,LPRS}, linear spaces~\cite{CDT2018}, partial orders~\cite{DGST} and sum-free sets~\cite{CJLWZ23,HJ2018,LSS2021}.

\section{Notation and Tools}\label{sec:notation}

In this section, we fix the notation and introduce basic concepts and results used in our proofs. 
Wherever convenient we identify a graph $G$ with its edge set, thus we write $e\in G$ instead of $e\in E(G)$.

The proof will make use of the following supersaturation result.
 \begin{theorem}\label{theorem:supersaturation}
For all $k\geq 3$ and $\eps>0$, there is a $\gamma>0$ and an $n_0$ such that any graph $G$ on $n>n_0$ vertices
with $e(G)>\left(1-\frac1{k-1}+\eps\right)\binom n2$ contains at least $\gamma n^k$ copies of $K_k$.\qed
\end{theorem}

Further we will need the following robust version of stability which is an extension of a result by F\"uredi~\cite{fu15}.
\begin{theorem}[{\cite[Theorem 1.2]{BBCLMS}}]
\label{thm:tfar}
For all positive integers $n,t,k$ the following holds. Every graph~$G$ which is not $t$-close to being~$k$-partite contains at least
\[\frac{n^{k-1}}{e^{2k}\cdot k!}\left(e(G)+t-\left(1-\frac1k\right)\frac{n^2}2 \right)\]
copies of $K_{k+1}$.\qed
\end{theorem}


\subsection{The hypergraph container theorem}

Given $\ell \geq 2$, an $\ell$-uniform hypergraph $\mathcal{H}=(V,E)$ is defined by a finite vertex set $V$ and a set $E \subset \{S \subset V \colon |S|=\ell\}$ of hyperedges. Let $v(\mathcal{H})$ and $e(\mathcal{H})$ denote the number of vertices and the number of edges of $\mathcal{H}$, respectively. The \emph{degree} of a vertex $v \in V$ and the \emph{co-degree} of a set $S \subset V$ are defined, respectively, as
$$d_{\mathcal{H}}(v)=\left|\{e \in E \colon v\in e\}\right| \textrm{ and }d_{\mathcal{H}}(S)=\left|\{e \in E \colon S \subseteq e\}\right|.$$
Clearly, the average degree $\bar{d}$ of $\mathcal{H}$ satisfies
$$\bar{d}(\mathcal{H})=\frac{1}{|V|}\sum_{v \in V} d_{\mathcal{H}}(v)=\frac{\ell e(\mathcal{H})}{v(\mathcal{H})}.$$
We also consider the \emph{maximum $j$-co-degree} of $\mathcal{H}$
$$\Delta_j({\mathcal{H}})=\max\{d_{\mathcal{H}}(S) \colon |S|=\ell\}$$
to define the following quantity for $0<\tau<1$:
\begin{equation}\label{codegree_function}
\Delta(\mathcal{H},\tau)=2^{\binom{\ell}{2}-1} \sum_{j=2}^\ell 2^{-\binom{j-1}{2}}\frac{\Delta_j(\mathcal{H})}{\bar{d}(\mathcal{H})\tau^{j-1}}.
\end{equation}
When the hypergraph $\mathcal{H}$ is clear from context, we will omit it to simplify the notation.  

A set $S \subset V$ is an independent set in $\mathcal{H}$ if $\mathcal{H}[S]$ contains no edges.  We shall use the following container theorem from~\cite{ST2015} (see also~\cite{BS18}).
\begin{theorem}\label{container_ST}
Let $\mathcal{H}$ be an $\ell$-uniform hypergraph with vertex set $[N]$. Fix $0<\varepsilon,\tau<1/2$ such that 
$\tau<1/(200 \cdot \ell \cdot \ell!^2)$ and $\Delta(\mathcal{H},\tau)\leq \varepsilon/(12 \cdot \ell!)$. Then there exists $c=c(\ell)<1000 \cdot \ell  \cdot \ell!^3$ and a collection $\mathcal{C}$ of vertex subsets with the following properties: 
\begin{itemize}
\item[(i)] Every independent set of $\mathcal{H}$ is a subset of some $C \in \mathcal{C}$.

\item[(ii)] For every $C \in \mathcal{C}$, $e(\mathcal{H}[C])\leq \varepsilon \cdot e(\HH)$.

\item[(iii)] $\displaystyle{\log{|\mathcal{C}|}} \leq cN\tau \cdot \log(1/\eps) \cdot \log(1/\tau)$.\qed
\end{itemize}

\end{theorem}

\subsection{Multicolor containers}

In this paper, we shall apply the container method in a way that was inspired by the work of Balogh and Li~\cite{BL19}. 
\begin{definition}[Template, subtemplate and rainbow copies]
An \emph{$r$-template} of a graph $G$ is a function $L\colon E(G) \rightarrow 2^{[r]}$ that associates with each $e\in G$ a list of colors $L(e)$. 
Given two $r$-templates $L_1$ and $L_2$ of $G$, we say that $L_1$ is a \emph{subtemplate} of $L_2$, denoted by $L_1\subset L_2$, if $L_1(e)\subset L_2(e)$ for every $e\in G$.

A \emph{rainbow copy of $K_4$ in a template $L$} is a set $\{(e_1,c_1),\ldots,(e_6,c_6)\}$ such that:
\begin{enumerate} 
\item The edges $e_1,\ldots,e_6$ produce a copy of $K_4$ in $G$.
\item The colors $c_1,\ldots,c_6$ are all distinct and satisfy $c_i\in L(e_i)$ for every $i\in[6]$.
\end{enumerate}
\end{definition}

Note that an edge coloring $\phi$ of a graph $G$ can be seen as a template of $G$ in which each $L(e)$ consists exactly of the color $\phi(e)$. 
With Theorem~\ref{container_ST}, we derive the following.
\begin{theorem}\label{thm_container}
For every $r \geq 6$, there exist $n_0=n_0(r)$ and a constant $c(r)$ such that the following hold. For any graph $G$ on $n$ vertices, $n \geq n_0$, there is a collection $\mathcal{C}$ of $r$-templates of $G$ such that:
\begin{enumerate}
\item\label{it:container1} Every rainbow-$K_4$-free $r$-coloring of $G$ is a subtemplate of some $L\in\mathcal{C}$.
\item \label{it:container2} For every $L \in \mathcal{C}$,  the number of rainbow-$K_4$ copies in $L$ is at most $r^{4} n^{-1/3}\binom{n}{4}$.
\item \label{it:container3} $\displaystyle{\left|\mathcal{C}\right|\leq 2^{c n^{-1/3} \binom{n}{2} \log^2{n}}}$.

\end{enumerate}
\end{theorem}

Indeed the notion  $r$-coloring can be replaced by $r$-template in \eqref{it:container1} but we do not need it here.





\begin{proof}
Let  $G$ be a graph on $n$ vertices,  $n$ sufficiently large.  Let $\HH=(V,E)$ be the hypergraph with vertex set $V=E(G)\times[r]$
which can be seen as the ``complete'' template $L$ of $G$ whose  lists of colors is $L(e)=[r]$ for each edge $e\in E(G)$.
We define the edges of $\HH$ to be the rainbow copies of $K_4$ in $L$.
Thus $\HH$ is  $6$-uniform and we will apply Theorem~\ref{container_ST} to $\mathcal{H}$. 

Using the notation of Theorem~\ref{container_ST}  we have   $N\leq r\binom{n}{2}$, $\ell=6$, $e(\HH)\leq r(r-1)\cdots (r-5)\binom{n}{4}$ and
$$\bar{d}\leq\binom{n-2}{2}(r-1)\cdots(r-5).$$
 With foresight, and for $n$ sufficiently large we set
\begin{equation}\label{eq_defs}
\eps=\frac{n^{-1/3}}{(r-1)(r-2)} <\frac12\textrm{ and }\tau= \sqrt{12 \cdot 6!} \cdot 2^9  n^{-1/3} <\frac1{1200 \cdot 6!^2}.
\end{equation}
We wish to show that $\Delta(\HH,\tau)$ defined in~\eqref{codegree_function} satisfies
\begin{equation}\label{eq1.1}
\Delta(\HH,\tau)= 2^{14} \left(\frac{\Delta_2}{\bar{d}\tau}+\frac{\Delta_3}{2\bar{d}\tau^2}+\frac{\Delta_4}{4\bar{d}\tau^3} +\frac{\Delta_5}{8\bar{d}\tau^4} +\frac{\Delta_6}{16\bar{d}\tau^5} \right)< \frac{\varepsilon}{12 \cdot 6!}.
\end{equation}
Clearly, $\Delta_6(\HH)=1$ and we now consider $\Delta_i$ for $2 \leq i \leq 5$. 

For $i=2$ consider $S=\{(e_1,c_1),(e_2,c_2)\}$. If $c_1=c_2$ or $e_1=e_2$, $d(S)=0$. Otherwise, $d(S)=(r-2)\cdots(r-5)$ if $e_1 \cap e_2=\emptyset$, and $d(S)=(n-3)(r-2)\cdots(r-5)$ if $|e_1\cap e_2|=1$. We conclude that 
$$\frac{\Delta_2}{\bar{d} \tau} \leq \frac{2 \cdot (n-3)(r-2)\cdots(r-5)}{(n-2)(n-3)(r-1)\cdots (r-5) \tau}<\frac{4}{(r-1)n\tau}.$$
For $i=3$ consider $S=\{(e_1,c_1),(e_2,c_2),(e_3,c_3)\}$. If two of the colors are the same, or if the three edges induce anything other than a triangle, a path on three edges or a star on three edges, then $d(S)=0$. Considering the other cases, we have
$$\frac{\Delta_3}{\bar{d} \tau^2} \leq \frac{2(n-3)(r-3)(r-4)(r-5)}{(n-2)(n-3)(r-1)\cdots (r-5) \tau^2}<\frac{4}{(r-1)(r-2)n\tau^2}.$$
With similar arguments we obtain 
\begin{eqnarray*}
\frac{\Delta_4}{\bar{d} \tau^3}&\leq&\frac{2(r-4)(r-5)}{(n-2)(n-3)(r-1)\cdots (r-5) \tau^3}<\frac{4}{(r-1)(r-2)(r-3)n^2\tau^3},\\
\frac{\Delta_5}{\bar{d} \tau^4}&\leq&\frac{2(r-5)}{(n-2)(n-3)(r-1)\cdots (r-5) \tau^4}<\frac{4}{(r-1)\cdots (r-4)n^2\tau^4}.\\
\end{eqnarray*}
Further, note that 
$$n^2\tau^5,n\tau^2 \ll n^2\tau^4 \qquad\text{and}\qquad  n\tau \ll n^2\tau^3$$
and as a consequence,
$$\Delta(\HH,\tau) \leq \frac{5 \cdot 2^{13} \Delta_3}{\bar{d} \tau^2}<\frac{5 \cdot 2^{15}}{(r-1)(r-2) n\tau^2} = \frac{n^{-1/3}}{(r-1)(r-2)} \frac{5 \cdot 2^{15}}{12 \cdot 6! \cdot 2^{18}}< \frac{\eps}{12 \cdot 6!}.$$
Applying Theorem~\ref{container_ST}, there is a family $\mathcal{C}$ of $r$-templates such that every rainbow-$K_4$-free $r$-coloring of $G$ is a subtemplate of some $L \in \mathcal{C}$. Further the number of rainbow copies of $K_4$ in any $L \in \mathcal{C}$ is at most
$$\eps e(\HH) \leq \frac{n^{-1/3}}{(r-1)(r-2)} \cdot  r(r-1)\cdots (r-5) \binom{n}{4} \leq n^{-1/3} r^4 \binom{n}{4}.$$
Finally, for  $c'<6000\cdot (6!)^3$, we have
$$|\mathcal{C}| \leq 2^{c' r \binom{n}{2} \tau \log(1/\eps) \log(1/\tau)}\leq 2^{c n^{-1/3} \binom{n}{2} \log^2{n}},$$
as required.
\end{proof}

\section{Proof of Theorem~\ref{thm_main2}}

In this section, we prove Theorem~\ref{thm_main2} which states that, for any $r\geq12$ and $\delta>0$, any sufficiently large $n$-vertex graph $G$ with $\rho_{r,4}(G)\geq  r^{\ex(n,K_4)}$ is $\delta n^2$-close to being tripartite. As mentioned in the introduction, it immediately gives the main result of this paper using Theorem~\ref{exact}. 

\begin{proof}[Proof of Theorem~\ref{thm_main2}]
Fix $r\geq 12$ and $\delta>0$, we may assume $\delta<1/2$. 
We apply Theorem~\ref{thm_container} to obtain $c_{\ref{thm_container}}$ and $n_{\ref{thm_container}}$. 
Let $\xi=\delta/300e^6$ 
and we apply Theorem~\ref{theorem:supersaturation} with $k=4$ and $\eps=\xi$ to obtain  $\gamma$ and $n_{\ref{theorem:supersaturation}}.$
We choose $n_0>\max\{n_{\ref{thm_container}}, n_{\ref{theorem:supersaturation}}\}$  sufficiently large.

Let $n>n_0$ and let $G$ be an $n$-vertex graph with $\rho_{r,4}(G) \geq r^{\ex(n,K_4)}$.
By applying Theorem~\ref{thm_container} to $G$ we obtain  a collection  $\mathcal{C}$ with properties stated in the theorem.
In particular,  property~\eqref{it:container1} guarantees that any  coloring $\phi$ of $G$ with no rainbow $K_4$ is a subtemplate $\phi\subset L'$ of some $L'\in \C$ and
by $\Phi(L')$ we denote the family of such colorings $\phi\subset L'$.
Then, by averaging there is a template $\widehat L$ such that 
\begin{equation*}
|\Phi(\widehat L)|\geq \frac{r^{\ex(n,K_4)}}{|\mathcal{C}|} \geq r^{\ex(n,K_4)-cn^{-1/3}\log^2(n) \binom{n}{2}}\geq r^{\frac{n^2}{3}-n^{7/4}}.
\end{equation*}

We modify $\widehat L$ to obtain an $L\supset \widehat L$ as follows: if $|\widehat L(e)|\geq 6$ we set $L(e)=[r]$ and we let $L(e)=\widehat L(e)$ otherwise.
Then 
\begin{equation}\label{eq_container}
 r^{\frac{n^2}{3}-n^{7/4}}\leq |\Phi(L)|\leq \prod_{e\in G}|L(e)|.
\end{equation}

\begin{claim}\label{cl:copiesL}
The number of rainbow  copies of $K_4$ in $L$ is at most $ r^{10} n^{-1/3} \binom{n}{4}$.
\end{claim}
\begin{proof}
By construction the underlying graph of a rainbow copy of $K_4$ in $\widehat L$ gives rise to at most $r^6$ rainbow copies of $K_4$ in $L$.
Further, by the property~\eqref{it:container2} of Theorem~\ref{thm_container} the template $\widehat L$ contains at most $r^{4} n^{-1/3} \binom{n}{4}$ rainbow copies of $K_4$.
Thus, it  suffices to show that the underlying graph of any rainbow copy of $K_4$ in $L$ also defines a rainbow copy of $K_4$ in $\widehat L$.

To this end consider the edges of a rainbow copy $F$ of $K_4$ in $L$ that is not a rainbow copy of $K_4$ in $\widehat L$. This means that the colors of some of the edges 
$e\in F$ are in $L(e)$ but not in $\widehat L(e)$. By construction, however, $|\widehat L(e)| \geq 6$ and there are at most six such edges. 
Thus these edges can be recolored using colors in $\widehat L(e)$ to obtain 
a rainbow copy of $K_4$ in $\widehat L$.
\end{proof}

Let $G_0$ be the spanning subgraph of $G$ obtained by removing all edges $e\in G$ with list of size $|L(e)|=1$ and let $n_0=n$. We will define a more structured subgraph of $G$ by 
successively performing the following two operations. The neighborhood of a vertex $v$ in a graph $G$ is denoted by $N_G(v)$.

\noindent \textbf{Operation 1.} If $G_i$ contains a vertex $v$ such that $\prod_{u\in N_{G_i}(v)} |L(uv)| \leq r^{\frac{(2-\xi^2)(n_i-1)}{3}}$ we let $G_{i+1}=G_i-v$ and $n_{i+1}=n_i-1$.

\begin{claim}\label{noop1}
We have 
\begin{eqnarray}\label{op1}
\prod_{e \in G_i}|L(e)|\leq r^{\frac{(2-\xi^2)(n_i-1)}{3}} \prod_{e \in G_{i+1}}|L(e)|
\end{eqnarray} 
and if Operation 1 is not applicable to $G_i$ then 
\begin{align}\label{op1'}\delta(G_i)>\frac{(2-\xi^2)(n_i-1)}{3}\qquad\text{and}\qquad |N^r_{G_i}(v)| >0.05 n_i\end{align}
where $N^r_{G_i}(v)$ denotes the neighbors $u$ of $v$ in $G_i$ such that $|L(uv)|=r$.
\end{claim}
\begin{proof}
The first property holds since \[\prod_{e \in G_i}|L(e)|=\prod_{u\in N_{G_i}(v)} |L(uv)| \prod_{e \in G_{i+1}}|L(e)|.\]
The first  part of \eqref{op1'} follows directly from the definition of Operation 1.
for the second part note that if the operation is not applicable to $G_i$
then
\[r^{\frac{(2-\xi^2)(n_i-1)}{3}}<\prod_{u\in N_{G_i}(v)} |L(uv)|\leq r^{|N^r_{G_i}(v)|} 5^{n_i-|N^r_{G_i}(v)|} \]
which yields
\[\frac{|N^r_{G_i}(v)|}{n_i} \geq   \frac{2/3- \xi^2/3-\log_{r}5}{1-\log_{r}5} \geq \frac{2/3- \xi^2/3-\log_{12}5}{1-\log_{12}5}  \stackrel{\xi \leq 10^{-2}}{>}0.05,\]
as claimed. \end{proof}
For the next operation we need the following notion.
\begin{definition}[Critical triangle]
A triangle $T$ in a graph $G_i\subset G$ is critical if it lies in at least $n^{5/6}$ rainbow copies  $F\subset G_i$ of $K_4$ in $L$.
\end{definition}


\noindent \textbf{Operation 2.} If $G_i$ contains a non-critical triangle $T$ with vertices $u,v,w$ and edges $e_1,e_2,e_3$ with $|L(e_1)|\geq |L(e_2)|\geq |L(e_3)|\geq 2$
such that 
\begin{enumerate} \item $|N_{G_i}(u) \cap N_{G_i}(v) \cap N_{G_i}(w)|\geq 19 \xi^2 (n_i-3)$ and 
\item $|L(e_1)|=r$ and $|L(e_2)| \geq 3$, 
 \end{enumerate}
 then let $G_{i+1}=G_i-\{u,v,w\}$ and $n_{i+1}=n_i-3$. 
 
 We first show a property analogous to~\eqref{op1}.
 \begin{claim}\label{op2}
Suppose that $n_i > \xi^2 n$ and $G_{i+1}$ is obtained by applying Operation 2 to $G_i$, then
\begin{align}\prod_{e \in G_i}|L(e)|\leq  r^{\frac{(2-\xi^2)((n_i-3)+(n_i-2)+(n_i-1))}{3}} \prod_{e \in G_{i+1}}|L(e)|\end{align}
 \end{claim}
\begin{proof}
Define a partition of $V(G_{i+1})=U_1\cup U_2\cup U_3$ by 
\begin{align*}U_1&=\{x\in V(G_{i+1})\colon  \text{there is a rainbow copy of $K_4$ in $G_i[u,v,w,x]$}\},\\
U_2&=\{x\in V(G_{i+1})\colon  \text{$x$ is a joint neighbor of $u,v,w$  but $G_i[u,v,w,x]$ is rainbow-$K_4$-free}\}, \text{ and}\\
U_3&=\{x\in V(G_{i+1})\colon \text{$x$ has at most two neighbors in $\{u,v,w\}$}\}.
\end{align*}
Since $T$ is non-critical  and   
$n_i > \xi^2 n$ we have
\begin{align}\label{u1u2}|U_1|\leq n^{5/6} \qquad\text{and}\qquad  |U_2| \geq 19 \xi^2 (n_i-3)-|U_1|\geq 18 \xi^2(n_i-3).
\end{align}
By definition, if $x\in U_3$ the product of the sizes of the lists of the edges connecting $x$ to $\{u,v,w\}$ is at most $r^2$. 
Consider now a vertex $x\in U_2$. We  first argue that, out of  $L(xu), L(xv)$ and $L(xw)$, at most one has size $r$ and for a contradiction suppose that $|L(xu)|=|L(xv)|=r$. 
As $L(xw)$ and $L(e_3)$ have size at least two and $L(e_2)$ has size at least three we can choose  different colors for these three  edges. 
Since all remaining  edges $xu, xv$ and $e_1$ have $r$ colors this can be extended to a rainbow $K_4$, which yields a contradiction. 
By a similar argument, if $|L(xu)|\geq 5$, then $L(xu)$ and  $L(xv)$ have size at most  three and we deduce that $|L(xu)|\cdot|L(xv)|\cdot |L(xw)|\leq \max\{9r,5 \cdot 3^2,4^3\}=9r$.

From this we deduce
\begin{align*}
\prod_{e \in G_i}|L(e)|&\leq r^3\left(\prod_{|\{r,s\}\cap\{u,v,w\}|=1} |L(rs)|  \right) \left( \prod_{e \in G_{i+1}}|L(e)|\right)\nonumber \\
&\leq r^{3} r^{3|U_1|} (9r)^{|U_2|} r^{2|U_3|} \left( \prod_{e \in G_{i+1}}|L(e)|\right)\nonumber\\ 
&= 9^{|U_2|} r^{|U_1|} r^{2(n_i-3)-|U_2|+3} \left( \prod_{e \in G_{i+1}}|L(e)|\right) \nonumber \\
&\stackrel{\eqref{u1u2}}{\leq} \left(\frac{9}{r}\right)^{18 \xi^2 (n_i-3)}  r^{n^{5/6}} r^{2(n_i-3)+3} \left( \prod_{e \in G_{i+1}}|L(e)|\right) \nonumber\\
&\stackrel{9^{18}\leq r^{16}}{\leq} r^{n^{5/6}}r^{(2-2\xi^2)(n_i-3)+3} \left( \prod_{e \in G_{i+1}}|L(e)|\right) \nonumber \\
&\stackrel{\xi^2n<n_i}{\leq}r^{(2-3\xi^2/2)(n_i-3)+3} \left( \prod_{e \in G_{i+1}}|L(e)|\right) \nonumber \\
&\leq r^{\frac{(2-\xi^2)((n_i-3)+(n_i-2)+(n_i-1))}{3}} \left( \prod_{e \in G_{i+1}}|L(e)|\right).
\end{align*} 
\end{proof}

Let $G_0,\ldots,G_p$ be a sequence obtained by performing the above operations until $n_p=|V(G_p)|\leq \xi^2n$ or until none of the above operations may be performed on $G_p$.  
By~\eqref{eq_container},~\eqref{op1}, and~Claim~{\ref{op2}} we deduce that 
\begin{equation}\label{eq_UB}
r^{\frac{n^2}{3}-n^{7/4}}{\leq} \prod_{e \in G}|L(e)| \leq r^{\left(\frac{2-\xi^2}{3}\right)\sum_{j=1}^{n-n_p}(n-j)} \prod_{e \in G_p}|L(e)| \leq r^{\left(\frac{2-\xi^2}{3}\right)\frac{(n-n_p)^2}{2}} \prod_{e \in G_p}|L(e)|.
\end{equation}
In particular, we must have $n_p > \xi^2 n$, otherwise the upper bound~$\prod_{e \in G_p}|L(e)|\leq r^{\binom{n_p}{2}}$ leads to
\[{\frac{n^2}{3}-n^{7/4}}\leq  {\left(\frac{1}{3}-\frac{5\xi^2}{6}+\frac{2\xi^4}{3}\right)n^2+\binom{n_p}{2}}\leq {\frac{n^2}{3}-\frac{5\xi^2n^2}{6}+\frac{7\xi^4n^2}{6}}  {\leq} {\left(\frac{1}{3}-\frac{\xi^2}{3}\right)n^2}\]
which yields a contradiction.

\medskip

Since  $n_p > \xi^2 n$ we infer that  neither Operation 1 nor Operation 2 can be performed in $G_p$. To further clean up the graph $G_p$ let 
$X_3$ denote the set of critical triangles. 
Further, we say that an edge in $G_p$ is critical if lies in at least $n_p^{{11}/{12}}$ critical triangles and we call a vertex in $G_p$ critical if it lies in at least $n_p^{{23}/{12}}$ critical triangles.
Let $X_2$ and $X_1$ denote the set of critical edges and vertices, respectively.
By definition each critical triangle is contained in at least $n^{5/6}$ rainbow copies of $K_4$ in $L$ and, by Claim~\ref{cl:copiesL}, $L$ itself contains at most $r^{10} n^{-1/3} \binom{n}{4}$ such copies. Therefore $|X_3|\leq r^{10} n^{-1/6} \binom{n}{3}$ and thus
\begin{align}\label{x1x2}|X_1| \leq \frac{3 |X_3|}{n_p^{{23}/{12}} }\leq \frac{3 r^{10} n^{11/12}}{\xi^4}  \qquad\text{and}\qquad |X_2|\leq \frac{3\cdot |X_3|}{n_p^{{11}/{12}} }\leq \frac{3r^{10} n^{{23}/{12}}}{\xi^2}.\end{align}

Let $G'=G_p-X_1$ and
$n'=|V(G')|\geq (1-\xi^2)n_p$. Since Operation 1 is not applicable to $G_p$ we deduce from~\eqref{x1x2} and~Claim~\ref{noop1} that 
\begin{equation}\label{eqGprime}
\delta(G')\geq \delta(G_p)- |X_1| \geq  \frac{(2-2\xi^2)n_p}{3} \geq  \frac{(2-2\xi^2)n'}{3}.
\end{equation}

In the following we will work towards a corresponding maximum degree (for a slightly modified graph). First we show the following.
\begin{claim}\label{cl:G'}
Suppose that some $v\in V(G')$ satisfies $\deg_{G'}(v) \geq \frac{(2+{2\xi})n'}3$. Then actually $\deg_{G^{'}}(v) \geq (\frac56-3\xi^2)n_p$.
\end{claim}
\begin{proof}
Due to $\delta(G')\ge \frac{(2-2\xi^2)n'}{3}$ and $\deg_{G'}(v) \geq \frac{(2+{2\xi})n'}3$ the joint neighborhood of $v$ and any  $u,w\in V(G')$ is at least
\[\deg_{G'}(u)+\deg_{G'}(w)+\deg_{G'}(v)-2n'\geq  \frac23{(\xi-\xi^2)n'}\geq 19 \xi^2 (n_p-3).\]
In particular, if $uw$ is an edge in the $r$-neighborhood $N_{G'}^r(v)$, then  $\{u,v,w\}$ must be a critical triangle as it satisfies 
the two properties for  Operation 2 in $G_p$, yet this operation is not applicable in $G_p$.

From the above and knowing  that $G'$ does not contain critical vertices, we conclude that there are at most $n_p^{23/12}$ edges in $N^r_{G'}(v)$.
Moreover, by~Claim~\ref{op1} and \eqref{x1x2},  we have 
\[|N^{r}_{G'}(v)|\geq |N^{r}_{G_p}(v)| - |X_1| \geq 0.04 n_p,\] thus there exists a vertex  
$w \in N_{G'}^r(v)$ with degree at most $\frac{2 n_p^{23/12}}{0.04 n_p}=50n_p^{11/12}$ in $N_{G'}^r(v)$. 
 By~\eqref{eqGprime}, we have
$$ \frac{(2-2\xi^2)}{3}n' \leq d_{G'}(w) \leq 50n_p^{11/12}+(n'-|N^r_{G'}(v)|)$$
and direct calculations yield
\begin{align}\label{eqGr}
|N^r_{G_p}(v)|\leq |N^r_{G'}(v)|  +|X_1|\leq \frac{(1+2\xi^2)n'}{3}+50n_p^{11/12} + |X_1| \leq \frac{(1+4 \xi^2)n_p}{3}. 
\end{align}

Finally, since Operation 1 is not applicable to $v$ we obtain
\begin{eqnarray*}
r^{\frac{(2-\xi^2)n_p}{3}} <  \prod_{u\in N_{G_p}(v)}|L(uv)| &\leq& r^{|N^r_{G_p}(v)|}\cdot 5^{\deg_{G_p}(v)-|N^r_{G_p}(v)|}\leq \left(\frac{r}{5}\right)^{\frac{(1+4 \xi^2)n_p}{3}}\cdot 5^{\deg_{G_p}(v)}.
\end{eqnarray*}
Thus, as $\log_5r\geq \frac{3}{2}$ for $r\geq 12$ we deduce
\begin{align*}\label{eq:dg3}
 \deg_{G_p}(v) \geq \frac{(1+4 \xi^2)n_p}{3} + \frac{(1-5 \xi^2)n_p}{3} \log_5{r} \geq \frac{(5-7\xi^2)n_p}{6}
\end{align*}
and the claim follows from $\deg_{G'}(v)\geq \deg_{G_p}(v) -|X_1|$ and~\eqref{x1x2}.
\end{proof}
Equipped with this result we now show that $G'$ does not contain many  vertices with too high degree.
\begin{claim}\label{cl:degG'}
There are at most $\xi^2n'$ vertices $v$ in $G'$ which satisfy $ \deg_{G^{'}}(v) \geq \frac{(2+2\xi)n'}{3}$.
In particular,  $e(G_p)<\left(\frac{1}{3}+\xi \right)n_p^2.$
\end{claim}
\begin{proof} 
Since $G'=G_p-X_1$ the first part of the claim immediately implies the second part 
\begin{align*}
e(G_p) &\leq e(G')+n_p|X_1|\leq \frac{1}{2} \left((1-\xi^2)n'\left(\frac{2}{3}+\xi\right)n'+\xi^2n'^2 \right)+ n_p|X_1|\\
&\stackrel{\eqref{x1x2}}{\leq} \left(\frac{1}{3}+\frac{\xi}{2}+\frac{\xi^2}{6} \right) n_p^2 < \left(\frac{1}{3}+\xi \right)n_p^2.
\end{align*}

Let $B$ be the set of vertices mentioned in the statement of the claim. For a contradiction, suppose that $|B|>\xi^2n'$. We remove from $G'$
the set $X_2$ of critical edges to obtain the subgraph $G''\subset G'$.
Let $v \in B$  be a vertex incident to the  least number of removed edges and let 
$w\in N_{G'}^r(v)$ be an $r$-neighbor incident to the least number of removed edges.
Claim~\ref{cl:G'}, the assumption on $B$ and~\eqref{x1x2}    then yield
\[\deg_{G''}(v) \geq \deg_{G'}(v) - \frac{2 |X_2|}{|B|}\geq \left(\frac{5}{6}-3 \xi^2\right)n_p - \frac{2 |X_2|}{|B|}\geq \left(\frac{5}{6}-4 \xi^2 \right)n_p\] 
and Claim~\ref{noop1}, $\delta(G')\geq\frac{(2-2\xi^2)}{3}n'$ due to~\eqref{eqGprime}, $n'\geq(1-\xi^2)n_p$ and~\eqref{x1x2} yield
\[\deg_{G''}(w) \geq \deg_{G'}(w) - \frac{2 |X_2|}{|N^r_{G'}(v)|}\geq \frac{(2-2\xi^2)n'}{3}-\frac{2|X_2|}{0.05 n'}\geq \left(\frac{2}{3}- \xi^2 \right)n_p.\]
In particular,  $v$ and $w$ have at least $\deg_{G''}(v)+\deg_{G''}(w)-n_p\geq\frac{n_p}2-5\xi^2n_p$
common neighbors and all but at most $n_p^{11/12}$ of these vertices $u$ form a non-critical triangle  $\{u,v,w\}$, since $vw$ is a non-critical edge.
Due to $\delta(G_p)\geq\frac{(2-2\xi^2)}{3}n_p$ and  $\deg_{G'}(v)\geq \deg_{G_p}(v)\geq \left(\frac{5}{6}-3 \xi^2\right)n_p$ we have for each such $u$ that
\[|N_{G_p}(u)\cap N_{G_p}(v) \cap N_{G_p}(w)| \geq \deg_{G_p}(u)+\deg_{G_p}(v)+\deg_{G_p}(w)-2n_p \geq \frac{n_p}{6}-6 \xi^2 n_p >19 \xi^2 (n_p-3).\]
Hence we must have $|L(vw)|=r$ and $|L(vu)|=|L(wu)|=2$ since  Operation 2 was not applicable to $\{u,v,w\}$. 
Thus, the vertex $v$ in $G_p$ is incident to at least $\left(\frac{1}{2} - 6\xi^2\right)n_p$ edges of $G_p$ with  lists of size two. From the fact that  Operation 1 was not applicable to $G_p$ we conclude
\begin{align*}
    \displaystyle r^{\left(\frac{2-\xi^2}{3}\right)n_p}<\prod_{u\in N_{G_p}(v)}|L(uv)|&\leq 2^{\left(\frac{1}{2} - 6\xi^2\right)n_p} \cdot r^{n_p-\left(\frac{1}{2} - 6\xi^2\right)n_p}  \leq 2^{\frac{n_p}{2}}\cdot r^{\frac{n_p}{2}+6\xi^2n_p}.
\end{align*}
This resolves to $r^{\frac{1}{6}-7\xi^2}\leq 2^{\frac{1}{2}}$ which yields a contradiction for $r\geq 12$.
\end{proof}

Recall from~\eqref{eq_UB} that we have \[r^{\frac{n^2}{3}-n^{7/4}}  \stackrel{\eqref{eq_UB}}{\leq} r^{\left(\frac{2-\xi^2}{3}\right)\frac{(n-n_p)^2}{2}} \prod_{e \in G_p}|L(e)|.\] 
By Claim~\ref{cl:degG'} we have 
$\prod_{e \in G_p}|L(e)|\leq r^{e(G_p)}<r^{\left(\frac{1}{3}+\xi \right)n_p^2}$. Thus
\[\frac{n^2}{3}-n^{7/4}-\left(\frac{2-\xi^2}{3}\right)\frac{(n-n_p)^2}{2}<\left(\frac{1}{3}+\xi \right)n_p^2\]
which is equivalent to 
\begin{equation}\label{ineq}
\frac23n_p\left(n-n_p\left(1+\frac32\xi\right)\right)+\frac{\xi^2}6(n-n_p)^2<n^{7/4}.
\end{equation}
We claim that this gives $n_p>(1-2\xi)n$. Indeed, if we assume that $n_p \leq (1-2\xi)n$ (recall that we already have $n_p \geq \xi^2 n$), the left-hand side would be at least
\[\frac23 \xi^2n^2\left(1-(1-2\xi)\left(1+\frac32 \xi\right)\right)  \geq \frac{\xi^3 n^2}{3},\]
contradicting~\eqref{ineq}. 

Moreover, let $m_i=\left|\{e \in G \colon | L(e)|=i\}\right|$ and let $m=m_2+m_3+\dots+m_5$. Note that 
\[\prod_{e \in G_p}|L(e)|=\prod_{i=1}^r i^{m_i} \leq 5^{m}r^{e(G_p)-m}\leq \left(\frac 5r\right)^mr^{\left(\frac{1}{3}+\xi \right)n^2}\]
and thus~\eqref{eq_UB} yields 
\[r^{\frac{n^2}{3}-n^{7/4}} \leq r^{\left(\frac{2-\xi^2}{3}\right)2\xi^2n^2}  \left(\frac 5r\right)^mr^{\left(\frac{1}{3}+\xi \right)n^2},\]
from which we derive
\[m_2+\dots+m_5=m \leq \left(\xi + \frac{4\xi^2}{3}\right)n^2 \left(\frac{1}{1-\log_5{r}}\right)<3\xi n^2.\]

Thus,  all but at most $3\xi n^2$ edges $e\in G_p$ have list  $L(e)$ of size at least six and clearly
a copy of  $K_4$ in $G_p$ using only these edges forms rainbow $K_4$ in $L$. 
By~\eqref{it:container2} of Theorem~\ref{thm_container} we infer that 
$G_p$ contains at most  $3\xi n^4+r^4n^{-1/3}\binom n4\leq 4\xi n^4$ copies of $K_4$.
Moreover, 
$e(G_p)\geq\left(\frac23-\xi\right)\binom{n_p}2$
by~\eqref{op1'} of Claim~\ref{noop1}. Hence, Theorem~\ref{thm:tfar} implies that $G_p$ is $10^4\xi n^2$-close to tripartite, i.e.,
there is a partition $U_1\cup U_2\cup U_3=V(G_p)$ such that \[e_{G_p}(U_1)+e_{G_p}(U_2)+e_{G_p}(U_3)<25 e^6\cdot\xi n^2.\]

\begin{claim}
Let $E_1$ denote the set of edges $e\in E_G(U_1)\cup E_G(U_2)\cup E_G(U_3)$ and $|L(e)|=1$. Then 
$|E_1|\leq 180 e^6\xi n^2$.
\end{claim}
\begin{proof}
Let $\widehat G\subset G_p$ be the subgraph of $G_p$ obtained by removing  all edges in $U_1, U_2,U_3$ and all edges $e\in G_p$ with $|L(e)|\leq 5$.
Then \[e(\widehat G)> e(G_p)- 25e^6\xi n^2-3\xi n^2\geq \frac23\binom{n_p}2 - 26e^6\xi n^2.\]   
Let $B_i$ be the largest bipartite subgraph of $E_1(U_i)$. Then for some 
 $i\in[3]$ we have $|B_i|\geq |E_1|/6$ and  we add
it to $\widehat G$. 
If $|E_1|> 180 e^6\xi n^2$ then $\widehat G \cup B_i$ contains at least $\frac23\binom{n_p}2+2e^6\xi n^2$ edges 
and by supersaturation, Theorem~\ref{theorem:supersaturation},
we would obtain at least $\delta_{\ref{theorem:supersaturation}}n^4>r^{10}n^{-1/3}\binom n4$ copies of $K_4$ in $\widehat G \cup B_i$. However, each of these copies 
has five edges in $\widehat G$ and one edge in $B_i$ hence it corresponds to a rainbow copy of $K_4$ in $L$. This yields a contradiction 
to the number of rainbow copies of $K_4$ in $L$ given in~Claim~\ref{cl:copiesL}. 
\end{proof}
Finally, since $G_p$ is obtained from $G$ by removing the edges in $E_1$ and then removing a vertex set $X$ of size $|X|<2\xi n$
the tripartition $U_1\cup U_2\cup (U_3\cup X)$ satisfies
\begin{align*}e_G(U_1)+e_G(U_1)+e_G(U_3\cup X)&<e_{G_p}(U_1)+e_{G_p}(U_2)+e_{G_p}(U_3)+|E_1|+|X||U_3|\\
&<25 e^6\cdot\xi n^2+ 180e^6\xi n^2+2\xi n^2<\delta n^2.
\end{align*}
as required.
\end{proof}


\begin{thebibliography} {PPP}
\bibitem{ABKS}
N.~Alon, J.~Balogh, P.~Keevash,  and B.~Sudakov, The number of edge colorings with no monochromatic cliques, \emph{J. London Math. Soc.} {\bf 70(2)} (2004), 273--288.

\bibitem{ADLRY94} N. Alon, R. A. Duke, H. Lefmann, V. R\"{o}dl, and R. Yuster, The algorithmic aspects of the regularity
lemma, \emph{J. Algorithms} {\bf 16} (1994), 80--109.


\bibitem{Bal06} J.~Balogh,
\newblock \emph{A remark on the number of edge colorings of graphs},
\newblock {European Journal of Combinatorics} {\bf 27} (2006), 565--573.


\bibitem{BBCLMS}    J. Balogh, N. Bushaw, M. Collares, H. Liu, R. Morris, and M. Sharifzadeh,  
\newblock \emph{The typical structure of graphs with no large cliques},
\newblock{ Combinatorica} (2016), 1–16.

\bibitem{BL19} J.~Balogh and L.~Li, 
\newblock \emph{The typical structure of Gallai colorings and their extremal graphs},
\newblock {SIAM Journal on Discrete Mathematics} {\bf 33(4)} (2019), 2416--2443.

\bibitem{BS18} J.~Balogh and J.~Solymosi,
\emph{On the number of points in general position
in the plane}, Discrete Analysis {\bf 16} (2018), 20pp.

\bibitem{BHS17} F.S.~Benevides, C.~Hoppen and  R.M.~Sampaio,
\newblock \emph{Edge-colorings of graphs avoiding a prescribed coloring pattern},
\newblock {Discrete Mathematics} {\bf 240(9)} (2017), 2143--2160.

\bibitem{botler} F. Botler, J. Corsten, A. Dankovics, N. Frankl, H. H\`{a}n, A. Jim\'{e}nez, and J. Skokan, \emph{Maximum number of triangle-free edge colourings with five and six colours}, Proceedings EUROCOMB 2019, Acta Mathematica Universitatis Comenianae {\bf88(3)} (2019), 495--499.


\bibitem{CJLWZ23} Y. Cheng, Y. Jing, L. Li, G. Wang, and W. Zhou, \emph{Integer colorings with forbidden rainbow sums}, Journal of Combinatorial Theory, Series A
{\bf 199} (2023), 105769.

\bibitem{CDT2018} D. Clemens, S. Das, and T. Tran, \emph{Colourings without monochromatic disjoint pairs}, {European Journal of Combinatorics} {\bf 70} (2018), 99--124.

\bibitem{CGM20} L. Colucci, E. Gy\H{o}ri and A. Methuku, 
\newblock \emph{Edge colorings of graphs without monochromatic stars}
\newblock Discrete Mathematics {\bf 343(12)} (2020), 112140.

\bibitem{DGST} S. Das, R. Glebov, B. Sudakov, and T. Tran, \emph{Colouring set families without monochromatic $k$-chains}, {J. Combinatorial Theory, Series A} {\bf 168} (2019),  84--119. 

\bibitem{Erd74}  P.~Erd{\H{o}}s, 
  Some new applications of probability methods to combinatorial
              analysis and graph theory,
 Proc.\,of the Fifth Southeastern Conference on
              Combinatorics, Graph Theory and Computing (Florida Atlantic
              Univ., Boca Raton, Fla., 1974),  Congressus Numerantium, No. {\bf X}, Utilitas Math. (1974), 39--51.


\bibitem{fu94} Z. F\"{u}redi, Extremal hypergraphs and combinatorial geometry, \emph{Proceedings of the International
Congress of Mathematicians}, Vol\. 1, 2 (Z\"{u}rich, 1994), 1343--1352, Birkh\"{a}user, Basel, 1995.

\bibitem{fu15} Z.~F\"uredi, A proof of the stability of extremal 
graphs, Simonovits's stability from Szemer\'edi's regularity,
\emph{Journal of Combinatorial Theory, Series B} {\bf 115} (2015), 66--71.

\bibitem{GS04} A.~Gy\'{a}rf\'{a}s, and G.~Simonyi G., \emph{Edge colorings of complete graphs without tricolored triangles}, Journal of Graph Theory {\bf 46} (2004), 211--216.

\bibitem{GPPS25+} P. Gupta, Y. Pehova, E. Powierski, and K. Staden, \emph{A framework for the generalised Erd\H{o}s-Rothschild problem and a resolution of the dichromatic triangle case}, arXiv 25032.12291 (2025), 57pp. 

\bibitem{hanjime} H. H\`{a}n, and A. Jim\'{e}nez, \emph{Improved bound on the maximum number of clique-free colorings with two and three colors}, SIAM Journal on Discrete Mathematics {\bf 32(2)} (2018), 1364--1368.

\bibitem{HJ2018} H. H\`{a}n, and A. Jim\'{e}nez, \emph{Maximum number of sum-free colorings in finite Abelian groups}, {Israel Journal of Mathematics} {\bf 226(2)} (2018), 505--534.

\bibitem{kneser} C. Hoppen, Y. Kohayakawa, and H. Lefmann, \emph{Hypergraphs with many Kneser colorings}, {European Journal of Combinatorics} {\bf 33} (2012), 816--84.

\bibitem{multipattern} C. Hoppen, H. Lefmann, and D.A. Nolibos, \emph{An extension of the rainbow Erd\H{o}s-Rothschild problem}, Discrete Mathematics {\bf 344(8)} (2021), 112443.

\bibitem{forbm} C.~Hoppen, Y.~Kohayakawa, and H.~Lefmann, \emph{Edge colorings of graphs avoiding monochromatic matchings of a given size}, Combinatorics, Probability \& Computing {\bf 21} (2012), 203--218.

\bibitem{linear} C.~Hoppen, Y.~Kohayakawa and H.~Lefmann,
\newblock \emph{Edge-colorings of graphs avoiding fixed monochromatic subgraphs with linear Tur\'{a}n number},
\newblock {European Journal of Combinatorics} {\bf 35(1)} (2014), 354--373.



\bibitem{rainbow_kn} C.~Hoppen, H.~Lefmann, and K.~Odermann,
\newblock \emph{A rainbow Erd\H{o}s-Rothschild problem},
\newblock {SIAM Journal on Discrete Mathematics}  {\bf 31} (2017), 2647--2674.

\bibitem{rainbow_triangle} C.~Hoppen, H.~Lefmann, and K.~Odermann,  \emph{On graphs with a large number of edge-colorings avoiding a rainbow triangle,} {European Journal of Combinatorics} {\bf 66} (2017), 168--190.


\bibitem{LPRS} H.~Lefmann, Y.~Person, V.~R\"odl, and M.~Schacht, \emph{On colourings of hypergraphs without monochromatic Fano planes,} Combinatorics, Probability \& Computing {\bf 18} (2009), 803--818.

\bibitem{LSS2021} H. Liu, M. Sharifzadeh, and K. Staden, \emph{On the maximum number of integer colourings with forbidden monochromatic sums}, Electronic Journal of Combinatorics {\bf 28(1)} (2021), 1--35.

\bibitem{PS2021} O. Pikhurko, and K. Staden, \emph{Stability for the Erd\H{o}s-Rothschild problem}, Forum of Mathematics, Sigma, {\bf 11}, 2023, e23, 1--50.

\bibitem{PS2022} O. Pikhurko, and K. Staden, \emph{Exact solutions to the Erd\H{o}s-Rothschild problem}, Forum of Mathematics, Sigma, {\bf 12}, 2024, e8, 1--45.

\bibitem{yilma} O. Pikhurko,  Z.B. Yilma, \emph{The maximum number of $K_3$-free and $K_4$-free edge $4$-colorings}, J. Lond. Math. Soc. {\bf 85} (2012), 593--615.

\bibitem{ST2015} D. Saxton and A. Thomason, \emph{Hypergraph containers}, \emph{Inventiones Mathematicae}, {\bf 201(3)} (2015), 925--992. 


 
\bibitem{turan} P.~Tur\'{a}n, {On an extremal problem in graph theory} (in Hungarian), \emph{Matematikai \'{e}s Fizikai Lapok} {\bf 48} (1941), 436--452.

\bibitem{yuster}
R.~Yuster, \emph{The number of edge colorings with no monochromatic triangle}, J. Graph Theory {\bf 21} (1996), 441--452.

\end{thebibliography}
\end{document}